\documentclass[10pt,oneside,a4paper,reqno]{amsart}  

\usepackage{amssymb,graphicx,mathrsfs,xcolor}
\usepackage[colorlinks=true, linkcolor=magenta, citecolor=cyan, urlcolor=blue]{hyperref}

\numberwithin{equation}{section}\newtheorem{theorem}{Theorem}[section]
\newtheorem{corollary}[theorem]{Corollary}\newtheorem{lemma}[theorem]{Lemma}

\theoremstyle{remark}
\newtheorem{remark}{Remark}[section]
\theoremstyle{definition}

\newcommand{\bra}[1]{\langle #1 \rangle}

\newcommand{\p}{\widetilde{p}}
\newcommand{\q}{\widetilde{q}}
\newcommand{\s}{\widetilde{s}}
\newcommand{\R}{\widetilde{r}}

\title[Angular integrability]
{Stein-Weiss and Caffarelli-Kohn-Nirenberg\\ 
inequalities with angular integrability}
\date{\today}    

\author{Piero D'Ancona}
\address{Piero D'Ancona: SAPIENZA - Unversit\`a di Roma,
Dipartimento di Matematica, Piazzale A.~Moro 2, I-00185 Roma, Italy}
\email{dancona@mat.uniroma1.it}

\author{Renato Luca'}
\address{Renato Luca': SAPIENZA - Unversit\`a di Roma,
Dipartimento di Matematica, Piazzale A.~Moro 2, I-00185 Roma, Italy}
\email{luca@mat.uniroma1.it}

\subjclass[2000]{%
}\keywords{}

\begin{document}\maketitle
\begin{abstract}
  We prove an extension of the Stein-Weiss weighted estimates
  for fractional integrals, in the context of $L^{p}$
  spaces with different integrability properties in the
  radial and the angular direction. 
  In this way, the classical estimates can be unified with their
  improved radial versions. A number of consequences
  are obtained: in particular we deduce precised
  versions of weighted Sobolev embeddings,
  Caffarelli-Kohn-Nirenberg estimates, and Strichartz
  estimates for the wave equation, which extend the radial
  improvements to the case of arbitrary functions.
\end{abstract}

\section{Introduction}

The radial estimate of Walter Strauss
\cite{Strauss77-a} states that for radial functions
$u\in \dot H^{1}(\mathbb{R}^{n})$, $n\ge2$, one has
\begin{equation}\label{eq:strauss}
  |x|^{\frac{n-1}{2}}|u(x)|\le 
  C\|\nabla u\|_{L^{2}},\qquad|x|\ge1.
\end{equation}
This is an example of a well known general phenomenon:
under suitable assumptions of symmetry, notably
radial symmetry, classical estimates and embeddings of spaces
admit substantial improvements. In the case of \eqref{eq:strauss},
a control on the $H^{1}$ norm of $u$ gives a pointwise
bound and decay of $u$, which are false in the general case.
Radial and more general symmetric estimates have been extensively
investigated, in view of their relevance for applications,
especially to differential equations.

This phenomenon is quite natural; indeed, symmetric functions
can be regarded as functions defined on lower dimensional manifolds,
hence satisfying stronger estimates, extended by the action of
some group of symmetries. Radial functions are essentially functions
on $\mathbb{R}^{+}$, while the norms on $\mathbb{R}^{n}$ introduce
suitable dimensional weights connected to the volume form.

In view of the gap between the symmetric and the non symmetric
case, an
interesting question arises: is it possible to quantify the
defect of simmetry of functions and prove more general
estimates which encompass all cases, and in particular
reduce to radial estimates when applied to radial functions?
Heuristically, one should be
able to improve on the general case by introducing some
measure of the distance from the maximizers of the
inequality, which typically have the greatest symmetry.

The aim of this paper is to give a partial 
positive answer to this question, through the use of
the following type of mixed radial-angular norms:
\begin{equation*}
  \|f\|_{L^{p}_{|x|}L^{\p}_{\theta}}=
  \left(
    \int_{0}^{+\infty}
    \|f(\rho\ \cdot\ )\|^{p}_{L^{\p}(\mathbb{S}^{n-1})}
    \rho^{n-1}d \rho
  \right)^{\frac1p},\qquad
  \|f\|_{L^{\infty}_{|x|}L^{\p}_{\theta}}=
  \sup_{\rho>0}\|f(\rho\ \cdot\ )\|_{L^{\p}(\mathbb{S}^{n-1})}.
\end{equation*}
When the context is clear we shall write simply $L^{p}L^{\p}$.
For $p=\p$ the norms reduce to the usual $L^{p}$ norms
\begin{equation*}
  \|u\|_{L^{p}_{|x|}L^{p}_{\theta}}\equiv
  \|u\|_{L^{p}(\mathbb{R}^{n})},
\end{equation*}
while for radial functions the value of $\p$ is irrelevant:
\begin{equation*}
  \text{$u$ radial}\quad\implies\quad 
  \|u\|_{L^{p}L^{\p}}\simeq \|u\|_{L^{p}(\mathbb{R}^{n})}
  \quad \forall p,\p\in[1,\infty].
\end{equation*}
Notice also that the norms are increasing in $\p$.
The idea of distinguishing radial and angular directions 
is not new and
has proved successful in the context of
Strichartz estimates and dispersive equations
(see \cite{MachiharaNakamuraNakanishi05-a},
\cite{Sterbenz05-a},
\cite{DanconaCacciafesta11-a}; see also
\cite{ChoOzawa09-a}). To give a flavour of the results which can
be obtained, Strauss' estimate \eqref{eq:strauss} can be extended
as follows:
\begin{equation*}
  |x|^{\frac np-\sigma}|u(x)|\lesssim
  \||D|^{\sigma}u\|_{L^{p}L^{\p}},\qquad
  \frac {n-1}\p+\frac1p<\sigma <\frac np
\end{equation*}
for arbitrary non radial functions $u$ and all $1<p<\infty$,
$1\le\p\le \infty$ (see Subsection \ref{sub:sobolev} below
for details and more general results).

A central role in our approach will be
played by the fractional integrals
$$
(T_{\gamma}\phi)(x)=\int_{\mathbb{R}^{n}}
  \frac{\phi(y)}{|x-y|^{\gamma}}dy, \qquad 0<\gamma<n.
$$
Weighted $L^{p}$ estimates for $T_{\gamma}$
are a fundamental problem of harmonic analysis, with a wide
range of applications. Starting from the
classical one dimensional case studied by Hardy and Littlewood,
an exhaustive analysis has been made of the admissible classes of
weights and ranges of indices
(see \cite{Stein93-a} and the references therein).
In the special case of power weights the optimal result is
due to Stein and Weiss:

\begin{theorem}[\cite{SteinWeiss58-b}]\label{SteinWeissThm}
Let $n\geq 1$ and $1< p\le q<\infty$.
Assume $\alpha,\beta,\gamma$ satisfy the set of conditions
($1=1/p+1/p'$)
\begin{equation}\label{eq:condSW}
\begin{split}
  &\beta<\frac nq,\quad \alpha<\frac{n}{p'},\quad 0<\gamma<n
    \\
  &\alpha+\beta+\gamma=n+\frac nq-\frac np
    \\
  &\alpha+\beta\ge0.
\end{split}
\end{equation}
Then the following inequality holds 
\begin{equation}\label{eq:stw}
  \||x|^{-\beta}T_{\gamma}\phi\|_{L^{q}}\le 
  C(\alpha,\beta,p,q)\cdot
  \||x|^{\alpha}\phi\|_{L^{p}}.
\end{equation}
\end{theorem}

Conditions in
the first line of \eqref{eq:condSW} are necessary to ensure
integrability,
while the necessity of 
the condition on the second line is due to scaling.
On the other hand, the sharpness of
$\alpha+\beta\ge0$ is less obvious and follows from the results
of \cite{SawyerWheeden92-a}.

In the radial case the last condition
can be relaxed and $\alpha+\beta$ is allowed to assume negative
values. Radial improvements were noticed in
\cite{Vilela01-a}, 
\cite{HidanoKurokawa08-a}, and the sharp result was obtained
recently by De Napoli, Dreichman and Dur\'an:

\begin{theorem}[\cite{DenapoliDrelichmanDuran09-a}]\label{DeNapoli1Thm}
  Let $n,p,q,\alpha,\beta,\gamma$ be as in the statement
  of Theorem \eqref{SteinWeissThm} but with the condition
  $\alpha+\beta\ge0$ relaxed to
  \begin{equation}\label{eq:condDDD}
    \alpha+\beta\ge(n-1)\left(\frac1q-\frac1p\right).
  \end{equation}
  Then estimate \eqref{eq:stw} is valid for all
  radial functions $\phi=\phi(|x|)$.
\end{theorem}

Using the $L^{p}_{|x|}L^{\p}_{\theta}$ norms we are able prove
the following general result which extends both theorems:

\begin{theorem}\label{the:Our1Thm}
  Let $n \geq 2$ and $1<p\le q<\infty$, $1\le\p\le\q\le\infty$. Assume
  $\alpha,\beta,\gamma$ satisfy the set of conditions
  \begin{equation}\label{eq:cDL}
  \begin{split}
    &\beta<\frac nq,\quad \alpha<\frac{n}{p'},\quad 0<\gamma<n
      \\
    &\alpha+\beta+\gamma=n+\frac nq-\frac np
      \\
    &\alpha+\beta\ge(n-1)
      \left(\frac1q-\frac1p+\frac{1}{\p}-\frac{1}{\q}\right).
  \end{split}
  \end{equation}
  Then the following estimate holds: 
  \begin{equation}\label{oHLS}
    \||x|^{-\beta}T_{\gamma} \phi\|
        _{L^{q}_{|x|}L^{\q }_{\theta}} 
    \le C
    \| |x|^{\alpha} \phi\|_{L^{p}_{|x|}L^{\p }_{\theta}}.
  \end{equation}
  The range of admissible $p,q$ indices 
  can be relaxed to $1\le p\le q\le \infty$ in two cases:
  \begin{enumerate}\setlength{\itemindent}{-0pt}
  \renewcommand{\labelenumi}{\textit{(\roman{enumi})}}
    \item when the third inequality in \eqref{eq:cDL} is strict, or
    \item when the Fourier transform $\widehat{\phi}$ has support
    contained in an annulus $c_{1}R\le|\xi|\le c_{2} R$
    ($c_{2}\ge c_{1}>0$, $R>0$);
    in this case \eqref{oHLS} holds with a constant independent 
    of $R$.
  \end{enumerate}
\end{theorem}

\begin{remark}
Notice that:
\begin{enumerate}\setlength{\itemindent}{-0pt}
\renewcommand{\labelenumi}{(\alph{enumi})}
  \item with the choices $q=\q $ and $p=\p $ (i.e.~in the usual 
  $L^{p}$ norms)
  Theorem \ref{the:Our1Thm} reduces to Theorem \ref{SteinWeissThm};
  \item if $\phi$ is radially symmetric, with the choice $\q =\p $,
  Theorem \ref{the:Our1Thm} reduces to Theorem \ref{DeNapoli1Thm}.
  Indeed, if $\phi$ is radially symmetric then
  $T_{\gamma}\phi$ is radially symmetric too, so that all choices
  for $\q,\p $ are equivalent;
  \item obviously, the same estimate is true for general 
  operators $T_{F}$ with nonradial kernels $F(x)$ satisfying
  \begin{equation*}
    T_{F}\phi(x)=\int F(x-y) \phi(y)dy,
    \qquad |F|\le C|x|^{-\gamma}.
  \end{equation*}
\end{enumerate}
\end{remark}

The proof of Theorem \ref{the:Our1Thm} is based on
two successive applications of Young's inequality 
for convolutions on suitable
Lie groups: first we use the strong inequality on the rotation
group $SO(n)$; then we use a Young inequality in the radial
variable, which in some cases must be replaced by 
the weak Young-Marcinkewicz
inequality on the multiplicative group $(\mathbb{R}^{+},\cdot)$
with the Haar measure $d\rho/\rho$. The convenient idea of using
convolution in the measure $d\rho/\rho$ was introduced in
\cite{DenapoliDrelichmanDuran09-a}.

\begin{remark}\label{rem:nonhomogeneous}
  The operator $T_{\gamma}$ is a convolution with the homogenous
  kernel $|x|^{-\gamma}$. Consider instead
  the convolution with a nonhomogeneous kernel
  \begin{equation*}
    S_{\gamma}\phi(x)=\int \frac{\phi(y)}{\bra{x-y}^{\gamma}}dy.
  \end{equation*}
  By the obvious pointwise bound
  \begin{equation*}
    |S_{\gamma}\phi(x)|\le T_{\gamma}|\phi|(x)
  \end{equation*}
  it is clear that $S_{\gamma}$ satisfies the same estimates as
  $T_{\gamma}$. However the scaling invariance of the estimate is
  broken, and indeed something more can be proved, thanks to the
  smoothness of the kernel (see Lemma \ref{lem:xmu}):
\end{remark}

\begin{corollary}\label{cor:nonhom}
  Let $n \geq 2$ and $1\le p\le q\le\infty$, 
  $1\le\p\le\q\le\infty$. Assume
  $\alpha,\beta,\gamma$ satisfy the set of conditions
  \begin{equation}\label{eq:condDL}
    \beta<\frac nq,\qquad \alpha<\frac{n}{p'},\qquad
    \alpha+\beta\ge(n-1)
      \left(\frac1q-\frac1p+\frac{1}{\p}-\frac{1}{\q}\right),
  \end{equation}
  \begin{equation}\label{eq:condabg}
    \alpha+\beta+\gamma>n\left(1+\frac1q-\frac1p\right).
  \end{equation}
  Then the following estimate holds: 
  \begin{equation}\label{ourHLS}
    \||x|^{-\beta}S_{\gamma} \phi\|
        _{L^{q}_{|x|}L^{\q }_{\theta}} 
    \le C
    \| |x|^{\alpha} \phi\|_{L^{p}_{|x|}L^{\p }_{\theta}}.
  \end{equation}
\end{corollary}

From the basic 
Theorem \ref{the:Our1Thm} a large number of inequalities 
can be deduced, which extend several
important classical estimates; we list a few  examples in the 
following.

\subsection{Weighted Sobolev embeddings}\label{sub:sobolev}

Recalling the pointwise bound
\begin{equation}\label{eq:derivT}
  |u(x)|\le C T_{\lambda}(||D|^{n-\lambda}u|),\qquad
  0<\lambda<n
\end{equation}
where $|D|^{\sigma}=(-\Delta)^{\frac s2}$, we see that an immediate
consequence of \eqref{ourHLS} is the weighted Sobolev
inequality
\begin{equation}\label{eq:weightS}
  \||x|^{-\beta}u\|_{L^{q}L^{\q}}\lesssim
  \||x|^{\alpha}|D|^{\sigma}u\|_{L^{p}L^{\p}}
\end{equation}
provided $1<p\le q<\infty$, $1\le\p\le\q\le\infty$ and
\begin{equation}\label{eq:condDLsob}
\begin{split}
  &\beta<\frac nq,\quad \alpha<\frac{n}{p'},\quad 0<\sigma<n
    \\
  &\alpha+\beta=\sigma+\frac nq-\frac np
    \\
  &\alpha+\beta\ge(n-1)
    \left(\frac1q-\frac1p+\frac{1}{\p}-\frac{1}{\q}\right).
\end{split}
\end{equation}
As usual, if the last condition is strict we can take $p,q$
in the full range $1\le p\le q\le \infty$. For instance, this
implies the inequality
\begin{equation}\label{eq:infsob}
  |x|^{-\beta}|u(x)|\lesssim
  \||x|^{\alpha}|D|^{\sigma}u\|_{L^{p}L^{\p}}
\end{equation}
provided $1\le p\le \infty$ and
\begin{equation*}
  \begin{split}
    &\beta<0,\quad \alpha<\frac{n}{p'},\quad 0<\sigma<n
      \\
    &\alpha+\beta=\sigma-\frac np
      \\
    &\alpha+\beta>(n-1)
      \left(\frac{1}{\p}-\frac1p\right).
  \end{split}
\end{equation*}
If we choose $\alpha=0$ we have in particular
for $p\in(1,\infty)$, $\p\in[1,\infty]$
\begin{equation}\label{eq:partsob}
  |x|^{\frac np-\sigma}|u(x)|\lesssim
  \||D|^{\sigma}u\|_{L^{p}L^{\p}},\qquad
  \frac {n-1}\p+\frac1p<\sigma <\frac np.
\end{equation}
This extends
to the non radial case the radial inequalities in
\cite{Strauss77-a},
\cite{Ni82-a},
\cite{ChoOzawa09-a}
and many others; notice that in the radial case we can choose
$\p=\infty$ to obtain the largest possible range.
When $\sigma$ is an integer we can replace
the fractional operator $|D|^{\sigma}$ with usual derivatives;
see Corollary \ref{cor:integers} below for a similar argument.

By similar techniques it is possible to derive nonhomogeneous
estimates in terms of norms of type $\|\bra{D}^{\sigma}u\|_{L^{p}}$;
we omit the details.

\subsection{Critical estimates in Besov spaces}\label{sub:besov}  

Case (ii) in Theorem \ref{the:Our1Thm} is suitable for
applications to spaces defined via Fourier decompositions,
in particular Besov spaces. We recall the standard
machinery:
fix a $C^{\infty}_{c}$ radial function $\psi_{0}(\xi)$
equal to 1 for $|\xi|<1$ and vanishing for $|\xi|>2$,
define a Littlewood-Paley partition of unity via
$\phi_{0}(\xi)=\psi(\xi)-\psi(\xi/2)$,
$\phi_{j}(\xi)=\phi_{0}(2^{-j}\xi)$,
and decompose $u$ as $u=\sum_{j\in \mathbb{Z}}u_{j}$
where $u_{j}=\phi_{j}(D)u=\mathscr{F}^{-1}\phi_{j}(\xi)\mathscr{F}u$.
Then the homogeneous Besov norm $\dot B^{s}_{p,1}$ is defined as
\begin{equation}\label{eq:besovn}
  \|u\|_{\dot B^{s}_{p,1}}=
  \sum_{j\in \mathbb{Z}}2^{js}\|u_{j}\|_{L^{p}}.
\end{equation}
We can apply Theorem \ref{the:Our1Thm}-(ii)
to each component $u_{j}$ 
in the full range of indices $1\le p\le q\le \infty$,
with a constant independent of $j$.
By the standard trick
$\widetilde{u}_{j}=u_{j-1}+u_{j}+u_{j+1}$, 
$u_{j}=\phi_{j}(D)\widetilde{u}_{j}$ we obtain the estimate
\begin{equation}\label{eq:besovest}
  \||x|^{-\beta}T_{\gamma} u\|
      _{L^{q}_{|x|}L^{\q }_{\theta}} 
  \le C
  \sum_{j\in \mathbb{Z}}
  \| |x|^{\alpha} \widetilde{u}_{j}\|_{L^{p}_{|x|}L^{\p }_{\theta}}
\end{equation}
for the full range $1\le p\le q\le \infty$, $1\le \p\le \q\le \infty$,
with $\alpha,\beta,\gamma$ satisfying \eqref{eq:condDLsob}.
The right hand side can be interpreted as a weighted norm of
Besov type with different radial and angular integrability;
in the special case $\alpha=0$, $p=\p>1$ we obtain a standard
Besov norm \eqref{eq:besovn} and hence the estimate 
(with the optimal choice $\q=\p=p$)
reduces to
\begin{equation}\label{eq:besovtrue}
  \||x|^{-\beta}T_{\gamma} u\|
      _{L^{q}_{|x|}L^{p }_{\theta}} 
  \le C
  \|u\|_{\dot B^{0}_{p,1}}.
\end{equation}
This estimate is weaker than \eqref{ourHLS} when the
third condition in \eqref{eq:condDL} is strict, but
in the case of equality it gives a new estimate:
recalling \eqref{eq:derivT}, we have proved the following

\begin{corollary}\label{cor:besov}
  For all $1< p\le q\le \infty$ we have
  \begin{equation}\label{eq:truesob}
    \||x|^{\frac {n-1}p-\frac {n-1}q} u\|
      _{L^{q}_{|x|}L^{p }_{\theta}} \le C
    \|u\|_{\dot B^{\frac1p-\frac1q}_{p,1}}.
  \end{equation}
\end{corollary}

If we restrict \eqref{eq:truesob} to radial functions and $q=\infty$, we
obtain the well known radial pointwise estimate
\begin{equation}\label{eq:radialbes}
  |x|^{\frac {n-1}p}|u|\le C
  \|u\|_{\dot B^{1/p}_{p,1}}\qquad
  1< p<\infty
\end{equation}
(see \cite{ChoOzawa09-a}, \cite{SickelSkrzypczak00-a}).

\subsection{Caffarelli-Kohn-Nirenberg weighted interpolation inequalities}
\label{sub:ckn}  

Consider the family of inequalities on $\mathbb{R}^{n}$, $n\ge1$
\begin{equation}\label{eq:CKN}
  \||x|^{-\gamma}u\|_{L^{r}}\le C
  \||x|^{-\alpha}\nabla u\|^{a}_{L^{p}}
  \||x|^{-\beta}u\|^{1-a}_{L^{q}}.
\end{equation}
for the range of parameters
\begin{equation}\label{eq:rangeCKN}
  n\ge1,\qquad
  1\le p<\infty,\qquad
  1\le q<\infty,\qquad
  0<r<\infty,\qquad
  0< a\le1.
\end{equation}
Some conditions are immediately seen to be necessary for the validity
of \eqref{eq:CKN}:
to ensure local integrability we need
\begin{equation}\label{eq:intCKN}
  \gamma<\frac nr\qquad
  \alpha<\frac np\qquad
  \beta<\frac nq
\end{equation}
and by scaling invariance we need to assume
\begin{equation}\label{eq:scaCKN}
  \gamma-\frac nr=
  a\left(\alpha+1-\frac np\right)+
  (1-a)\left(\beta-\frac nq\right).
\end{equation}
In \cite{CaffarelliKohnNirenberg84-a} the following remarkable
result was proved, which improves and extends a number of earlier
estimates including weighted Sobolev and Hardy inequalities:

\begin{theorem}[\cite{CaffarelliKohnNirenberg84-a}]\label{the:CKN}
  Consider the inequalities \eqref{eq:CKN} 
  in the range of parameters given by
  \eqref{eq:intCKN}, \eqref{eq:rangeCKN}, \eqref{eq:scaCKN}.
  Denote with $\Delta$ the quantity
  \begin{equation}\label{eq:Delta}
    \Delta=\gamma-a \alpha-(1-a)\beta \equiv
    a +n
      \left(
        \frac1r-\frac{1-a}{q}-\frac ap
      \right)
  \end{equation}
  (the identity in \eqref{eq:Delta} is a reformulation of
  the scaling relation \eqref{eq:scaCKN}).
  Then the inequalities \eqref{eq:CKN} are true if and only if
  both the following conditions are satisfied:
  \begin{enumerate}\setlength{\itemindent}{-0pt}
  \renewcommand{\labelenumi}{\textit{(\roman{enumi})}}
    \item $\Delta\ge0$
    \item $\Delta\le a$ when
    $\gamma-n/r=\alpha+1-n/p$.
  \end{enumerate}
\end{theorem}

\begin{remark}\label{rem:original}
  Notice that in the original formulation of 
  \cite{CaffarelliKohnNirenberg84-a} also the case $a=0$
  was considered, but with the introduction of an additional parameter
  forcing $\beta=\gamma$ when $a=0$. Thus the case $a=0$
  becomes trivial
  in the original formulation; however, at least for $r>1$,
  a much larger range $0\le \gamma-\beta<n$ can be obtained
  by a direct application of the Hardy-Littlewood-Sobolev inequality,
  so strictly speaking the additional requirement $\beta=\gamma$ is not
  necessary. We think the formulation adopted here is cleaner.
  
  On the other hand, the necessity of (i) follows from
  the uniformity of the estimate w.r.to translations, while the
  necessity of (ii) is proved by testing the inequality on
  the spikes $|x|^{\gamma-n/r}\log|x|^{-1}$ truncated near $x=0$.
\end{remark}

In \cite{DenapoliDrelichmanDuran11-a} 
the authors prove the
following radial improvement
of Theorem \ref{the:CKN}:

\begin{theorem}[\cite{DenapoliDrelichmanDuran11-a}]
\label{the:CKNDDD}
  Let $n\ge2$,
  let $\alpha,\beta,\gamma,r,p,q,a$ be in the range determined by
  \eqref{eq:intCKN}, \eqref{eq:rangeCKN}, \eqref{eq:scaCKN},
  define $\Delta$ as in \eqref{eq:Delta}, and assume that
  \begin{equation}\label{eq:CKNDDD}
    a\left(1-\frac np\right)\le \Delta\le a,\qquad
    \alpha<\frac np-1,
  \end{equation}
  the first inequality being strict
  when $p=1$. Then estimate \eqref{eq:CKN} is true for
  all radial functions $u\in C^{\infty}_{c}(\mathbb{R}^{n})$.
\end{theorem}

We somewhat simplified the statement of Theorem 1.1 in
\cite{DenapoliDrelichmanDuran11-a}, and in particular conditions
(1.8)-(1.10) in that paper are equivalent to \eqref{eq:CKNDDD} 
here, as it is
readily seen. Notice that 
the condition $\Delta\le a$ forces 
$r$ to be larger than 1.

Using the $L^{p}L^{\p}$ norms we can extend both
Theorems \ref{the:CKN} and \ref{the:CKNDDD}. For greater
generality we prove an estimate with fractional derivatives
\begin{equation*}
  |D|^{\sigma}=(-\Delta)^{\frac \sigma2},\qquad \sigma>0.
\end{equation*}
Our result is the following:

\begin{theorem}\label{the:Our2Thm}
  Let $n\ge2$, $r,\R,p,\p,q,\q\in[1,+\infty)$, $0<a\le1$,
  $0<\sigma<n$ with
  \begin{equation}\label{eq:Ico}
    \gamma<\frac nr,\qquad
    \beta<\frac nq,\qquad
    \frac np-n<\alpha<\frac np-\sigma
  \end{equation}
  satisfying the scaling condition
  \begin{equation}\label{eq:scalingCKN}
    \gamma-\frac nr=
    a\left(\alpha+\sigma-\frac nr\right)+
    (1-a)\left(\beta-\frac nq\right).
  \end{equation}
  Define the quantities
  \begin{equation}\label{eq:Deltas}
    \Delta=a \sigma+n
      \left(
        \frac1r-\frac{1-a}{q}-\frac ap
      \right),\qquad
    \widetilde{\Delta}=a \sigma+n
      \left(
        \frac1\R-\frac{1-a}{\q}-\frac a\p
      \right).
  \end{equation}
  and assume further that
  \begin{equation}\label{eq:IIco}
    \Delta+(n-1)\widetilde{\Delta}\ge0,
  \end{equation}
  \begin{equation}\label{eq:IIIco}
    1<p,\qquad
    a\left(\sigma-\frac np\right)<\Delta\le a \sigma, \qquad 
    a\left(\sigma-\frac n\p\right)\le \widetilde{\Delta}\le a \sigma.
  \end{equation}
  Then the following interpolation inequality holds:
  \begin{equation}\label{eq:CKNnostra}
    \||x|^{-\gamma}u\|_{L^{r}_{|x|}L^{\R}_{\theta}}\le C
    \||x|^{-\alpha}|D|^{\sigma} u\|^{a}_{L^{p}_{|x|}L^{\p}_{\theta}}
    \||x|^{-\beta}u\|^{1-a}_{L^{q}_{|x|}L^{\q}_{\theta}}.
  \end{equation}
  If one assumes strict inequality in \eqref{eq:IIco}, then
  the inequalities in \eqref{eq:IIIco} can be relaxed to 
  non strict inequalities.
\end{theorem}

When $\sigma$ is an integer, the condition on $\alpha$
from below can be dropped, and a slightly stronger estimate 
can be proved.
We introduce the notation
\begin{equation*}
  \||x|^{-\alpha}D^{\sigma} u\|_{L^{p}L^{\p}}=
  \sum_{|\nu|=\sigma}
  \||x|^{-\alpha}D^{\nu} u\|_{L^{p}L^{\p}},\qquad
  \nu=(\nu_{1},\dots,\nu_{n})\in \mathbb{N}^{n}.
\end{equation*}
Then we have:

\begin{corollary}\label{cor:integers}
  Assume $\sigma=1,\dots,n-1$ is an integer. Then
  the following estimate holds
  \begin{equation}\label{eq:CKNnostraint}
    \||x|^{-\gamma}u\|_{L^{r}_{|x|}L^{\R}_{\theta}}\le C
    \||x|^{-\alpha}D^{\sigma} u\|^{a}_{L^{p}_{|x|}L^{\p}_{\theta}}
    \||x|^{-\beta}u\|^{1-a}_{L^{q}_{|x|}L^{\q}_{\theta}}.
  \end{equation}
  provided the parameters satisfy the same conditions as in the
  previous theorem, with the exception of the condition
  $\alpha>-n+n/p$
  which is not necessary.
\end{corollary}

\begin{remark}\label{rem:comparCKN}
  If $\sigma=1$, Corollary \ref{cor:integers} contains both the
  original result of \cite{CaffarelliKohnNirenberg84-a}
  (for $\Delta\le a$)
  and the radial improvement of \cite{DenapoliDrelichmanDuran09-a}.
  
  Indeed,
  if we choose $p=\p$, $q=\q$, $r=\R$ in Corollary \ref{cor:integers}
  we get of course $\Delta=\widetilde{\Delta}$, and selecting
  $\sigma=1$ we reobtain the original inequality 
  \eqref{eq:CKN} in the range $0\le \Delta\le a$.
  
  On the other hand, if $u$ is a radial function, estimate 
  \eqref{eq:CKNnostraint} does not
  depend on the choice of $\p,\q,\R$ and we can let $\widetilde{\Delta}$
  assume an arbitrary value in the range \eqref{eq:IIIco}.
  Thus if $\Delta>a(\sigma-n/p)$ we can choose $\widetilde{\Delta}=0$,
  while if $\Delta=a(\sigma-n/p)$ we can choose
  $\widetilde{\Delta}=\epsilon>0$ arbitrarily small, recovering
  the results of Theorem \ref{the:CKNDDD}.
\end{remark}

The classical application of CKN estimates is to the regularity
of solutions to the Navier-Stokes equation; this will be the subject
of forthcoming papers.

\subsection{Strichartz estimates for the wave equation}\label{sub:WE}  

As a last example, we
mention an application of our result to Strichartz
estimates for the wave equation; a more detailed
analysis will be conducted elsewhere.
The wave flow $e^{it|D|}$ on $\mathbb{R}^{n}$, $n\ge2$,
satisfies the estimates,
which are usually called \emph{Strichartz estimates}:
\begin{equation}\label{eq:strich}
  \||D|^{\frac nr+\frac1p-\frac n2}e^{it|D|}f\|_{L^{p}_{t}L^{r}_{x}}
  \lesssim \|f\|_{L^{2}}
\end{equation}
provided the indices $p,r$ satisfy
\begin{equation}\label{eq:admWE}
  p\in[2,\infty],\qquad
  0<\frac1r\le\frac12-\frac{2}{(n-1)p}.
\end{equation}
Here the $L^{p}_{t}L^{r}_{x}$ norms are defined as
\begin{equation*}
  \|u(t,x)\|_{L^{p}_{t}L^{r}_{x}}=
  \left\|
     \|u(t,\cdot)\|_{L^{r}_{x}}
  \right\|_{L^{p}_{t}}.
\end{equation*}
In their most general version, the estimates were proved in 
\cite{GinibreVelo95-b},
\cite{KeelTao98-a}. Notice that in \eqref{eq:strich} we included
the extension of the estimates which can be obtained via
Sobolev embedding on $\mathbb{R}^{n}$. 

If the initial value $f$ is a radial function, the estimates admit
an improvement in the sense that conditions \eqref{eq:admWE}
can be relaxed to
\begin{equation}\label{eq:admradial}
  p\in[2,\infty],\qquad
  0<\frac1r<\frac12-\frac{1}{(n-1)p}.
\end{equation}
This phenomenon is connected with the finite speed of
propagation for the wave equation and is usually deduced 
using the space-time decay properties of the equation.
For a thourough discussion and a comprehensive history
of such estimates see 
e.g.~\cite{JiangWangYu10-a} and the references therein.

A different set of estimates are the \emph{smoothing estimates},
also known as Morawetz-type or weak dispersion estimates.
These appear in a large number of versions;
a particularly sharp one is the following, from
\cite{FangWang08-a}:
\begin{equation}\label{eq:smooWE}
  \||x|^{-\zeta}|D|^{\frac12-\zeta}
     e^{it|D|}f\|_{L^{2}_{t}L^{2}_{x}}
  \lesssim \|\Lambda^{\frac12-\zeta} f\|_{L^{2}},\qquad
  \frac12<\zeta <\frac n2.
\end{equation}
Here the operator
\begin{equation*}
  \Lambda=(1-\Delta_{\mathbb{S}^{n-1}})^{1/2}
\end{equation*}
is a function of the Laplace-Beltrami operator on the
sphere and acts only on angular variables, thus we see
that the flow improves the angular regularity.
Morawetz-type estimates are conceptually simpler than
\eqref{eq:strich}, being related to more basic properties
of the operators; indeed $L^{2}$ estimates of this type
can be proved 
for quite large classes of equations
via multiplier methods.

Corresponding estimates are known for the Schr\"odinger flow
$e^{it \Delta}$, and
M.C.~Vilela \cite{Vilela01-a} noticed that in the radial case
they can be used to
deduce Strichartz estimates via the radial Sobolev embedding. 
Following a similar idea for the wave flow, 
in combination with our precised 
estimates \eqref{eq:weightS}, gives an even better result,
which strengthens the standard Strichartz estimates
\eqref{eq:strich}-\eqref{eq:admWE} in
terms of the mixed $L_{|x|}^{p}L_{\theta}^{\p}$ norms.
Indeed, a special case of \eqref{eq:weightS} gives, 
for arbitrary functions $g(x)$,
\begin{equation}\label{eq:special}
  \|g\|_{L^{q}_{|x|}L^{\q}_{\theta}}\lesssim
  \||x|^{\alpha}|D|^{\alpha+\frac n2-\frac nq}g\|_{L^{2}},\qquad
  q,\q\in[2,\infty),\qquad
  \frac n2>\alpha\ge(n-1)\left(\frac1q-\frac1\q\right)
\end{equation}
with the exclusion of the case $\alpha=0$, $q=\q=2$. Then by
\eqref{eq:special} and \eqref{eq:smooWE}
we obtain the precised Strichartz estimates
\begin{equation}\label{eq:prestrich}
  \||x|^{-\delta}|D|^{\frac nq+\frac12-\frac n2-\delta}e^{it|D|}f\|
      _{L^{2}_{t}L^{q}_{|x|}L^{\q}_{\theta}}
  \lesssim \|\Lambda^{-\epsilon} f\|_{L^{2}}
\end{equation}
provided
\begin{equation}\label{eq:condpre}
  q,\q\in[2,+\infty),\qquad
  \delta<\frac nq,\qquad
  0<\epsilon<\frac{n-1}{2},\qquad
  0<\frac1q<\frac1\q-\frac{1}{2(n-1)}
\end{equation}
and
\begin{equation}\label{eq:condpre2}
  \epsilon\le \delta+(n-1)\left(\frac1\q-\frac{1}{2(n-1)}-\frac1q\right).
\end{equation}

\section{Proof of Theorem \ref{the:Our1Thm}}

The first result we need is an explicit estimate of 
the angular part of the fractional integral
$T_{\gamma}\phi$. Notice that a similar analysis
in the radial case was done in
\cite{DenapoliDrelichmanDuran09-a} (see Lemma 4.2 there).
The following estimates are sharp:

\begin{lemma}\label{lem:singint}
  Let $n\ge2$, $\nu>0$, and write
  $\bra{x}=(1+|x|^{2})^{1/2}$.
  Then the integral
  \begin{equation*}
    I_{\nu}(x)=\int_{\mathbb{S}^{n-1}}|x-y|^{-\nu }dS(y)
    \qquad x\in \mathbb{R}^{n}
  \end{equation*}
  satisfies
  \begin{equation}\label{eq:stima0}
    |I_{\nu}(x)|\simeq\bra{x}^{-\nu }\qquad
    \text{for}\quad|x|\ge2,
  \end{equation}
  while for $|x|\le2$ we have
  \begin{equation}\label{stimaI}
    |I_{\nu}(x)| \simeq
    \left\{ 
    \begin{array}{cc}
      1& 
            \mbox{if} \ \ \nu <n-1 \\
      |\log{||x|-1|}| + 1& 
            \mbox{if} \ \ \nu =n-1 \\
     ||x|-1|^{n-1- \nu }  & 
             \mbox{if} \ \ \nu  > n-1.
    \end{array} 
    \right. 
  \end{equation}
\end{lemma}

\begin{proof}
We consider four different regimes according to the size of
$|x|$. We write for brevity $I$ instead of $I_{\nu}$.

\subsection*{First case: $|x|\geq 2$} 
For $x$ large and $|y|=1$
we have $|x-y| \simeq |x|$, 
hence $|I(x)| \simeq |x|^{-\nu } 
\simeq \langle x \rangle^{-\nu }$.
This proves \eqref{eq:stima0}.

\subsection*{Second case: $0 \leq |x| \leq \frac{1}{2}$} 
Clearly we have $|x-y| \simeq 1$ when $|y|=1$, 
and this implies 
$|I(x)| \simeq 1 \simeq \langle x \rangle^{-\nu }$.
This is equivalent to \eqref{stimaI} when $|x|\le1/2$.

\subsection*{Third case: $1 \leq |x| \leq 2$} 
This is the bulk of the computation since it contains the
singular part of the integral, as $|x|\to1$.
We write the integral in polar coordinates using the
spherical angles $(\theta_{1},\theta_{2},...,\theta_{n-1})$ 
on $\mathbb{S}^{n-1}$, oriented in such a way that
$\theta_{1}$ is the angle between $x$ and $y$.
Using the notation $\sigma=|x-y|$,
by the symmetry of $I(x)$ in $(\theta_{2},...,\theta_{n-1})$ we have
$$
  |I(x)| \simeq \int_{0}^{\pi}
  \sigma^{-\nu }(\sin{\theta_{1}})^{n-2}d\theta_{1}.
$$
In order to rewrite the integral using $\sigma$ as a new
variable, we compute 
$$
  2\sigma d\sigma = d(|x-y|^{2})=d(|x|+1 -2|x| \cos{\theta_{1}})
  =2|x|\sin{\theta_{1}}d\theta_{1}
$$
so we have
$$
  (\sin{\theta_{1}})^{n-2}d\theta_{1} = 
  \frac{\sigma (\sin{\theta_{1}})^{n-3}}{|x|}d\sigma
$$
and, noticing that $0\le|x|-1\le|x-y|=\sigma\le|x|+1$,
$$
  |I(x)| \simeq \int_{|x|-1}^{|x|+1}
  \sigma^{1-\nu }\frac{(\sin{\theta_{1}})^{n-3}}{|x|}d\sigma.
$$
Now let $A$ be the area of the triangle with vertices
$0,x$ andd $y$: we have $2A=|x|\sin{\theta_{1}}$ so that
$$
  |I(x)| \simeq  |x|^{2-n}\int_{|x|-1}^{|x|+1}
  \sigma^{1-\nu }A^{n-3}d \sigma.
$$
Recalling Heron's formula for the area of a triangle as a function
of the length of its sides we obtain
$$
  |I(x)| \simeq |x|^{2-n}\int_{|x|-1}^{|x|+1}
  \sigma^{1-\nu }
  \Bigl[(|x|+\sigma +1)(|x|+\sigma -1)
  (|x|+1 -\sigma)(\sigma +1 -|x|)\Bigr]^{\frac{n-3}{2}}d\sigma.
$$
Notice that this formula is correct for all dimensions $n\ge2$.

Now we split the integral as $I \simeq I_{1} + I_{2}$
with
$$
  I_{1}(x)= |x|^{2-n}\int_{|x|-1}^{|x|}
  \sigma^{1-\nu }
  \Bigl[
  (|x|+\sigma +1)
  (|x|+\sigma -1)
  (|x|+1 -\sigma)
  (\sigma +1 -|x|)
  \Bigr]^{\frac{n-3}{2}}
  d\sigma
$$
and
$$
  I_{2}(x) = |x|^{2-n}\int_{|x|}^{|x|+1}
  \sigma^{1-\nu }
  \Bigl[
  (|x|+\sigma +1)
  (|x|+\sigma -1)
  (|x|+1 -\sigma)
  (\sigma +1 -|x|)
  \Bigr]^{\frac{n-3}{2}}
  d\sigma.
$$
In the second integral $I_{2}$, recalling that $1\le|x|\le2$,
we have
\begin{equation*}
  |x|\simeq
  \sigma \simeq 
  |x|+\sigma+1 \simeq
  |x|+\sigma-1 \simeq
  \sigma+1-|x| \simeq
  1
\end{equation*}
so that
\begin{equation*}
  I_{2}\simeq
  \int_{|x|}^{|x|+1}(|x|+1-\sigma)^{\frac{n-3}{2}}d \sigma
  = \int_{0}^{1}(1-\sigma)^{\frac{n-3}{2}}d \sigma
  \simeq 1.
\end{equation*}
In the first integral $I_{1}$, using that
$1 \leq |x| \leq 2$ and $|x|-1 \leq \sigma \leq |x|$,
we see that 
\begin{equation*}
  |x|\simeq
  (|x|+\sigma +1)\simeq
  (|x|+1 -\sigma)\simeq
  1;
\end{equation*}
moreover,
\begin{equation*}
  1\le \frac{|x|+\sigma-1}{\sigma}\le 2
  \quad\text{so that}\quad |x|+\sigma-1 \simeq \sigma
\end{equation*}
and we have
$$
  I_{1}(x) \simeq \int_{|x|-1}^{|x|}
  \sigma^{1-\nu  + \frac{n-3}{2}}
  (\sigma+1 -|x|)^{\frac{n-2}{2}}d\sigma
$$
or, after the change of variable $\sigma\to\sigma(|x|-1)$,
$$
  I_{1}(x)\simeq(|x|-1)^{n-1-\nu   }
  \int_{1}^{1+\frac{1}{|x|-1}}
  (\sigma -1)^{\frac{n-3}{2}}
  \sigma ^{\frac{n-1}{2}-\nu }d\sigma .
$$
Now split the last integral as $A+B$ where
$$
   A= (|x|-1)^{n-1-\nu }\int_{1}^{2}
   (\sigma -1)^{\frac{n-3}{2}}\sigma ^{\frac{n-1}{2}-\nu }d\sigma 
$$
and
$$
  B=(|x|-1)^{n-1-\nu }
  \int_{2}^{1+\frac{1}{|x|-1}}(\sigma -1)^{\frac{n-3}{2}}
  \sigma ^{\frac{n-1}{2}-\nu }d\sigma;
$$
we have immediately
$$  
  A \simeq (|x|-1)^{n-1 - \nu }
$$
while, keeping into account that $\sigma  \simeq \sigma -1$ 
for $\sigma$ in $(2,1+\frac{1}{|x|-1})$,
$$
  B=(|x|-1)^{n-1-\nu }
  \int_{2}^{1+\frac{1}{|x|-1}}
  \sigma ^{n-2-\nu }d\sigma
$$
which gives
\begin{equation}\label{andamentoB}
  B\simeq 
    \left\{ \begin{array}{cc}
     1& \mbox{if} \ \ \nu <n-1 \\
     |\log{||x|-1|}| + 1& \mbox{if} \ \ \nu =n-1 \\
     ||x|-1|^{n-1- \nu }  & \mbox{if} \ \ \nu  > n-1 
  \end{array} \right.
\end{equation}

\subsection*{Fourth case: $\frac{1}{2} \leq |x| \leq 1$} 
Using the change of variable $|x'|=1/|x|$, we see that
$|I(x)|\simeq |I(1/|x'|)|$, 
thus the fourth case follows immediately from the third one,
and this concludes the proof of the Lemma.
\end{proof}

We shall also need the following estimate which is
proved in a similar way:

\begin{lemma}\label{lem:singint2}
  Let $n\ge2$, $\nu>0$. Then the integral
  \begin{equation*}
    J_{\nu}(x,\rho)=\int_{\mathbb{S}^{n-1}}
    \bra{x-\rho\theta}^{-\nu}dS(\theta)
    \qquad x\in \mathbb{R}^{n},\ \rho\ge0
  \end{equation*}
  satisfies:
  \begin{equation}\label{eq:stimab0}
    |J_{\nu}(x,\rho)|\simeq\bra{x}^{-\nu }\qquad
    \text{for $\rho\le1$ or $|x|\ge 2\rho$},
  \end{equation}
  \begin{equation}\label{eq:stimac0}
    |J_{\nu}(x,\rho)|\simeq\bra{\rho}^{-\nu }\qquad
    \text{for $|x|\le1$ or $\rho\ge 2|x|$},
  \end{equation}
  while in the remaining case, i.e.~when $|x|\ge1$ and
  $\rho\ge1$ and $2^{-1}|x|\le\rho\le2|x|$,
  \begin{equation}\label{eq:stimabI}
    |J_{\nu}(x,\rho)| \simeq
    \left\{ 
    \begin{array}{cc}
      \bra{\rho}^{-\nu}& 
            \mbox{if} \ \ \nu <n-1 \\
      \bra{\rho}^{-\nu}
          \log\left(\frac{2\bra{\rho}}{\bra{|x|-\rho}}\right)
      & 
            \mbox{if} \ \ \nu =n-1 \\
      \bra{\rho}^{1-n}\bra{|x|-\rho}^{n-1-\nu}  & 
             \mbox{if} \ \ \nu  > n-1.
    \end{array} 
    \right. 
  \end{equation}
  As a consequence, one has
  $J_{\nu}\lesssim \bra{\rho+|x|}^{-\nu}$ when $\nu<n-1$ and
  $J_{\nu}\lesssim \bra{\rho+|x|}^{-\nu}\log(2\bra{\rho}+|x|)$ 
  when $\nu=n-1$.
\end{lemma}

\begin{proof}
  The proof is similar to the proof of Lemma \ref{lem:singint};
  we sketch the main steps. Estimates \eqref{eq:stimab0} and
  \eqref{eq:stimac0} are obvious, thus we focus on \eqref{eq:stimabI}.
  Write $r=|x|$, so that we are in the region $1/2\le r/\rho\le2$;
  we shall consider in detail the case
  \begin{equation*}
    1\le \frac{r}{\rho}\le 2,
  \end{equation*}
  the remaining region being similar.
  Using the same coordinates as before, the integral is reduced to
  \begin{equation*}
    J_{\nu}(|x|,\rho)=
    |x|^{2-n}\int_{|x|-1}^{|x|+1}\bra{\rho \sigma}^{-\nu}
    A^{n-3}\sigma\ d \sigma
  \end{equation*}
  where $A$ is given by Heron's formula
  \begin{equation*}
    A(|x|,\sigma)^{2}=
    (|x|+\sigma+1)
    (|x|+\sigma-1)
    (|x|+1-\sigma)
    (\sigma+1-|x|).
  \end{equation*}
  We split the integral on the intervals $|x|\le\sigma\le|x|+1$ 
  and $|x|-1\le\sigma\le|x|$. The first piece gives 
  \begin{equation*}
    I_{1}\simeq 
     \bra{\rho}^{-\nu}
      \int_{|x|}^{|x|+1}(|x|+1-\sigma)^{\frac{n-3}{2}}d \sigma
  \end{equation*}
  and by the change of variable $\sigma\to \sigma(|x|+1)$
  we obtain
  \begin{equation*}
    I_{1}(|x|,\rho)\simeq \bra{\rho}^{-\nu}.
  \end{equation*}
  For the second integral on $|x|-1\le\sigma\le|x|$, noticing that
  \begin{equation*}
    1\le \frac{|x|+\sigma-1}{\sigma}\le2
  \end{equation*}
  we have
  \begin{equation*}
  \begin{split}
    I_{2}\simeq &
    \int_{|x|-1}^{|x|}
      \bra{\rho \sigma}^{-\nu} \sigma^{\frac{n-1}{2}}
      (\sigma+1-|x|)^{\frac{n-3}{2}}d \sigma
    \\
    = &
    (|x|-1)^{n-1}
    \int_{1}^{\frac{|x|}{|x|-1}}
    \bra{(r-\rho)\sigma}^{-\nu}\sigma^{\frac{n-1}{2}}
    (\sigma-1)^{\frac{n-3}{2}}d \sigma
  \end{split}
  \end{equation*}
  via the change of variables $\sigma\to \sigma(|x|-1)$
  which gives $\rho \sigma\to(r-\rho)\sigma$.
  The part of the integral bewteen 1 and 2 produces
  \begin{equation*}
    \simeq
    (|x|-1)^{n-1}\bra{r-\rho}^{-\nu}=
    \rho^{1-n}(r-\rho)^{n-1}\bra{r-\rho}^{-\nu}
  \end{equation*}
  while the remaining part between 2 and $|x|/(|x|-1)$ gives
  \begin{equation*}
  \begin{split}
    \simeq &
    (|x|-1)^{n-1}
    \int_{2}^{\frac{r}{r-\rho}}
    \bra{(r-\rho)\sigma}^{-\nu} 
    \sigma^{n-2}d \sigma
    \\
    = &
    \rho^{1-n}
    \int_{2(r-\rho)}^{r}\bra{\sigma}^{-\nu}\sigma^{n-2}d \sigma
    \\
    \simeq\ &
    \rho^{1-n}
    \int_{2(r-\rho)}^{r}
    \frac{\sigma^{n-2}}{1+\sigma^{\nu}}d \sigma
  \end{split}
  \end{equation*}
  which can be computed explicitly. Summing up we obtain
  \eqref{eq:stimabI}.
\end{proof}

We are ready for the main part of the proof.
By the isomorphism 
\begin{equation*}
  \mathbb{S}^{n-1}\simeq SO(n)/SO(n-1)
\end{equation*}
we can represent integrals on $\mathbb{S}^{n-1}$ in the form
\begin{equation*}
  \int_{\mathbb{S}^{n-1}}g(y)dS(y)= c_{n}
  \int_{SO(n)}g(Ae)dA,\qquad n\ge2
\end{equation*}
where $dA$ is the left Haar measure on
$SO(n)$, and $e\in\mathbb{S}^{n-1}$ is a fixed arbitrary
unit vector. Thus, via polar coordinates, 
a convolution integral can be written as follows
(apart from inessential constants depending only on the
space dimension $n$):
\begin{equation*}
\begin{split}
  F*\phi(x)=
  \int_{\mathbb{R}^{n}}F(x-y)\phi(y)dy
  &=
  \int_{0}^{\infty}
  \int_{\mathbb{S}^{n-1}}
     F(x-\rho \omega)\phi(\rho \omega)dS_{\omega}\rho^{n-1}d \rho
    \\
  &\simeq
  \int_{0}^{\infty}
  \int_{SO(n)}F(x-\rho Be)\phi(\rho Be)dB\rho^{n-1}d \rho
\end{split}
\end{equation*}
Hence the $L^{\q}$ norm of the convolution
on the sphere can be written as
\begin{equation*}
\begin{split}
  \|F*\phi(|x|\theta)\|_{L^{\q}_{\theta}(\mathbb{S}^{n-1})}
  &\simeq
  \|F*\phi(|x|Ae)\|_{L^{\q}_{A}(SO(n))}
    \\
  &\le 
  \int_{0}^{\infty}
  \left\|
    \int_{SO(n)}F(|x|Ae-\rho Be)\phi(\rho Be)dB
  \right\|_{L^{\q}_{A}(SO(n))}
  \rho^{n-1}d \rho
\end{split}
\end{equation*}
where $e$ is any fixed unit vector. By the change of variables
$B\to AB^{-1}$ in the inner integral
(and the invariance of the measure) 
this is equivalent to
\begin{equation*}
  =\int_{0}^{\infty}
  \left\|
    \int_{SO(n)}F(AB^{-1}(|x|Be-\rho e))\phi(\rho AB^{-1}e)dB
  \right\|_{L^{\q}_{A}(SO(n))}
  \rho^{n-1}d \rho
\end{equation*}
If $F$ satisfies
\begin{equation}\label{eq:rad}
  |F(x)|\le C f(|x|)
\end{equation}
for a radial function $f$, we can write
\begin{equation*}
  |F(AB^{-1}(|x|Be-\rho e))|\le 
  C f\left(\bigl||x|Be-\rho e\bigr|\right)
\end{equation*}
and we notice that the integral
\begin{equation*}
  \int_{SO(n)}f\left(\bigl||x|Be-\rho e\bigr|\right)
     |\phi(\rho AB^{-1}e)|dB=
  g*h(A)
\end{equation*}
is a convolution on $SO(n)$ of the functions
\begin{equation*}
  g(A)=f\left(\bigl||x|Ae-\rho e\bigr|\right),\qquad
  h(A)=|\phi(\rho Ae)|.
\end{equation*}
We can thus apply the Young's inequality on $SO(n)$
(see e.g.~Theorem 1.2.12 in 
\cite{Grafakos08-a}) and we obtain, for any
\begin{equation*}
  \q,\R,\p \in[1,+\infty]
  \quad\text{with}\quad
  1+\frac{1}{\q}=\frac1\R+\frac1\p,
\end{equation*}
the estimate
\begin{equation}\label{eq:firstest}
  \|F*\phi(|x|\theta)\|_{L^{\q}_{\theta}(\mathbb{S}^{n-1})}
  \lesssim 
  \int_{0}^{\infty}
  \|f(||x|e-\rho\theta|)\|_{L^{\R}_{\theta}(\mathbb{S}^{n-1})}
  \|\phi(\rho\theta)\|_{L^{\p}_{\theta}(\mathbb{S}^{n-1})}
  \rho^{n-1}d \rho
\end{equation}
where we switched back to the coordinates of $\mathbb{S}^{n-1}$.
Notice that the conditions on the indices imply in particular
\begin{equation*}
  \q\ge\p.
\end{equation*}
Specializing $f$ to the choice
\begin{equation*}
  f(|x|)=|x|^{-\gamma}
\end{equation*}
we get
\begin{equation*}%
  \|F*\phi(|x|\theta)\|_{L^{\q}_{\theta}}
  \lesssim 
  \int_{0}^{\infty}
  \rho^{-\gamma}
  \||\rho^{-1}|x|e-\theta|^{-\gamma}\|_{L^{\R}_{\theta}}
  \|\phi(\rho\theta)\|_{L^{\p}_{\theta}}
  \rho^{n-1}d \rho
\end{equation*}
which can be written in the form
\begin{equation*}
  =
  |x|^{n-\alpha-\frac np-\gamma}
  \int_{0}^{\infty}
  \left(
  \frac{|x|}{\rho}
  \right)^{\alpha+\frac np-n+\gamma}
  \||\rho^{-1}|x|e-\theta|^{-\gamma}\|_{L^{\R}_{\theta}}
  \rho^{\alpha+\frac np}
  \|\phi(\rho\theta)\|_{L^{\p}_{\theta}}
  \frac{d\rho}{\rho}
\end{equation*}
or equivalently, recalling \eqref{eq:condSW},
\begin{equation*}
  =
  |x|^{\beta-\frac nq}
  \int_{0}^{\infty}
  \left(
  \frac{|x|}{\rho}
  \right)^{-\beta+\frac nq}
  \||\rho^{-1}|x|e-\theta|^{-\gamma}\|_{L^{\R}_{\theta}}
  \rho^{\alpha+\frac np}
  \|\phi(\rho\theta)\|_{L^{\p}_{\theta}}
  \frac{d\rho}{\rho}
\end{equation*}
Following \cite{DenapoliDrelichmanDuran09-a},
we recognize that the
last integral is a convolution
in the multiplicative group $(\mathbb{R},\cdot)$ with the Haar measure
$d \rho/\rho$, which implies
\begin{equation*}
  |x|^{-\beta+\frac nq}
  \|F*\phi(|x|\theta)\|_{L^{\q}_{\theta}}
  \lesssim
  g_{1}*h_{1}(|x|),
\end{equation*}
with
\begin{equation*}
  g_{1}(\rho)=\rho^{-\beta+\frac nq}
  \||\rho e-\theta|^{-\gamma}\|_{L^{\R}_{\theta}},\qquad
  h_{1}(\rho)=
  \rho^{\alpha+\frac np}
  \|\phi(\rho\theta)\|_{L^{\p}_{\theta}}.
\end{equation*}
By the weak Young's inequality in the measure $d\rho/\rho$
(Theorem 1.4.24 in \cite{Grafakos08-a})
we obtain
\begin{equation*}
\begin{split}
  \||x|^{-\beta}F*\phi\|_{L^{q}L^{\q}}\equiv
  &\left\||x|^{-\beta+\frac nq}
     \|F*\phi(|x|\theta)\|_{L^{\q}_{\theta}}
  \right\|_{L^{q}(\rho^{-1}d\rho)}
    \\
  \lesssim &
  \|h_{1}\|_{L^{p}(\rho^{-1}d\rho)}
  \|g_{1}\|_{L^{r,\infty}(\rho^{-1}d\rho)}
\end{split}
\end{equation*}
that is to say
\begin{equation}\label{eq:almostfin}
  \||x|^{-\beta}F*\phi\|_{L^{q}L^{\q}}
  \lesssim
  \|\phi\|_{L^{p}L^{\p}}
  \left\|
    \rho^{-\beta+\frac nq}\||\rho e-\theta|^{-\gamma}\|
       _{L^{\R}_{\theta}}
  \right\|_{L^{r,\infty}(\rho^{-1}d\rho)}.
\end{equation}
provided
\begin{equation*}
  q,r,p\in(1,+\infty)\qquad
  1+\frac1q=\frac1r+\frac1p.
\end{equation*}
In particular this implies
\begin{equation}\label{eq:qp}
  q>p.
\end{equation}
In order to achieve the proof, it remains to check that the
last norm in \eqref{eq:almostfin} is finite.
Notice that, when $\R<\infty$,
\begin{equation*}
  \||\rho e-\theta|^{-\gamma}\|_{L^{\R}_{\theta}}=
  I_{\gamma \R}(\rho e)^{\frac{1}{\R}}
\end{equation*}
where $I_{\nu}$ was defined and estimated in Lemma \ref{lem:singint}.
On the other hand, when $\R=\infty$ one has directly
\begin{equation}\label{eq:rinf}
  \R=\infty \quad\implies\quad
  \||\rho e-\theta|^{-\gamma}\|_{L^{\R}_{\theta}}\simeq
  |\rho-1|^{-\gamma}.
\end{equation}

Using cutoffs, we split the $L^{r,\infty}$ norm in three regions
$0\le \rho\le 1/2$, $\rho\ge2$ and $1/2\le \rho\le2$.

In the region $0\le\rho\le1/2$, recalling 
\eqref{eq:stima0}-\eqref{stimaI} or \eqref{eq:rinf},
we have
\begin{equation*}
  I_{\gamma \R}(\rho e)^{\frac{1}{\R}}\simeq 1
  \quad\implies\quad
  \rho^{-\beta+\frac nq}I_{\gamma \R}(\rho e)^{\frac{1}{\R}}
  \in L^{1}(0,1/2; d\rho/\rho)
\end{equation*}
since by assumption $\beta<n/q$; thus the contribution of this
part to the $L^{r,\infty}(d\rho/\rho)$ norm is finite.

In the region $\rho\ge2$ we have
\begin{equation*}
  I_{\gamma \R}(\rho e)^{\frac{1}{\R}}\simeq 
    \rho^{-\gamma}
  \quad\implies\quad 
  \rho^{-\beta+\frac nq}I_{\gamma \R}(\rho e)^{\frac{1}{\R}} \simeq
  \rho^{-\beta-\gamma+\frac nq} 
  \in L^{1}(2,\infty; d\rho/\rho)
\end{equation*}
since the condition
\begin{equation*}
  -\beta-\gamma+\frac nq<0 \quad\iff \quad
  \alpha<\frac{n}{p'}
\end{equation*}
is satisfied by \eqref{eq:condDL}, and again the
contribution to the $L^{r,\infty}$ norm is finite.

For the third region $1/2\le \rho\le2$, by estimate
\eqref{stimaI}, we see that in the case
$\gamma\R\le n-1$ one has again, for some $\sigma\ge0$,
\begin{equation*}
  I_{\gamma \R}(\rho e)^{\frac{1}{\R}}\simeq 
    |\log||\rho|-1|^{\sigma}
  \quad\implies\quad 
  \rho^{-\beta+\frac nq}I_{\gamma \R}(\rho e)^{\frac{1}{\R}}
  \in L^{1}(1/2,2; d\rho/\rho)
\end{equation*}
On the other hand, in the case $\gamma\R>n-1$ (which includes
the choice $\R=\infty$), we see that
\begin{equation*}
  \rho^{-\beta+\frac nq}I_{\gamma \R}(\rho e)^{\frac{1}{\R}}\simeq
  |\rho-1|^{\frac{n-1}{\R}-\gamma}
  \in L^{r,\infty}(1/2,2; d\rho/\rho)\quad
  \iff \frac{n-1}{\R}-\gamma\ge-\frac1r.
\end{equation*}
Recalling the relation between $q,r,p$ (resp.~$\q,\R,\p$) the last
condition is equivalent to
\begin{equation*}
  -\gamma
  \ge 
  (n-1)
  \left(
    \frac1q-\frac1p-\frac1\q+\frac1\p
  \right)-\frac nq+\frac np-n
\end{equation*}
which is precisely the third of conditions \eqref{eq:condDL}.

The weak Young inequality can be used in \eqref{eq:almostfin}
only in the range $q,r,p\in(1,+\infty)$, which forces
\begin{equation*}
  1<p<q<\infty.
\end{equation*}
To cover the cases
\begin{equation*}
  1\le p<q\le\infty
\end{equation*}
we use instead the strong Young inequality: we can write
\begin{equation}\label{eq:almostfin2}
  \||x|^{-\beta}F*\phi\|_{L^{q}L^{\q}}
  \lesssim
  \|\phi\|_{L^{p}L^{\p}}
  \left\|
    \rho^{-\beta+\frac nq}\||\rho e-\theta|^{-\gamma}\|
       _{L^{\R}_{\theta}}
  \right\|_{L^{r}(\rho^{-1}d\rho)}
\end{equation}
for the full range 
$q,r,p\in[1,+\infty]$. The previous arguments
are still valid apart from the last step which must be replaced by
\begin{equation*}
  \rho^{-\beta+\frac nq}I_{\gamma \R}(\rho e)^{\frac{1}{\R}}\simeq
  |\rho-1|^{\frac{n-1}{\R}-\gamma}
  \in L^{r}(1/2,2; d\rho/\rho)\quad
  \iff \frac{n-1}{\R}-\gamma>-\frac1r
\end{equation*}
and this implies that the inequality in the last condition
\eqref{eq:condDL} must be strict. 

The case
\begin{equation*}
  1<p=q<\infty
\end{equation*}
has already been covered. Indeed, in this case
the scaling condition \eqref{eq:condDL} implies
\begin{equation*}
  \alpha+\beta+\gamma=n
  \quad\implies\quad
  \alpha+\beta>0
\end{equation*}
since $\gamma<n$. Thus when $\p=\q$ the last inequality in
\eqref{eq:condDL} is strict and we can apply the second part of
the proof; the cases $\p\le\q$ follow from the case $\p=\q$.

To complete the proof, it remains to consider the case (ii) where
we assume that the support of the Fourier transform
$\widehat{\phi}$ is contained in an annular region of size $R$.
By scaling invariance of the inequality, it is sufficient to
consider the case $R=1$. Now let $\psi(x)$ be such that
$\widehat{\psi}\in C^{\infty}_{c}$ and precisely
\begin{equation*}
  \widehat{\psi}(\xi)=1 \quad\text{for $c_{1}'\le|\xi|\le c_{2}'$},
  \qquad
  \widehat{\psi}(\xi)=0 \quad\text{for $|\xi|>2c_{1}'$ and 
       $|\xi|<\frac12 c_{2}'$},
\end{equation*}
for some constants $c'_{2}>c_{2}\ge c_{1}>c'_{1}>0$. This implies
\begin{equation*}
  \phi =\mathscr{F}^{-1}(\widehat{\psi}\widehat{\phi})
  =\psi*\phi
\end{equation*}
and we can write
\begin{equation*}
  T_{\gamma}\phi=|x|^{-\gamma}*\psi*\phi=
  (T_{\gamma}\psi)*\phi.
\end{equation*}
Since $T_{\gamma}\psi=c\mathscr{F}^{-1}
(|\xi|^{\gamma-n}\widehat{\psi}(\xi))$ is a Schwartz class function,
we arrive at the estimates
\begin{equation}\label{eq:allN}
  |T_{\gamma}\phi(x)|\le C_{\mu,\gamma}\bra{x}^{-\mu}*|\phi|\qquad
  \forall \mu\ge1.
\end{equation}
Here we can take $\mu$ arbitrarily large. Thus the proof of case
(ii) is concluded by applying the following Lemma:

\begin{lemma}\label{lem:xmu}
  Let $n\ge2$.
  Assume $1\le p\le q\le \infty$,
  $1\le \p\le \q\le \infty$ and $\alpha,\beta, \mu$ satisfy
  \begin{equation}\label{eq:condabmu}
    \beta<\frac nq,\qquad
    \alpha<\frac{n}{p'},\qquad
    \alpha+\beta\ge(n-1)
    \left(\frac1q-\frac1p+\frac1\p-\frac1\q\right),
  \end{equation}
  \begin{equation}\label{eq:condmu}
    \mu>
    -\alpha-\beta+n\left(1+\frac1q-\frac1p\right).
  \end{equation}
  Then the following estimate holds:
  \begin{equation}\label{eq:estmu}
    \||x|^{-\beta}\bra{x}^{-\mu}*\phi\|
       _{L^{q}_{|x|}L^{\q}_{\theta}}\lesssim
    \|\phi\|_{L^{p}_{|x|}L^{\p}_{\theta}}.
  \end{equation}
\end{lemma}

\begin{proof}
  Notice that, by \eqref{eq:condabmu},
  the right hand side in \eqref{eq:condmu} is always strictly
  positive and never larger than 
  $n-1$, thus it is sufficient to prove the lemma for $\mu$
  in the range
  \begin{equation*}
    0<\mu\le n.
  \end{equation*}
  By \eqref{eq:firstest} we have, for all $\p,\q,\R\in[1,+\infty]$ with
  $1+1/\q=1/\R+1/\p$,
  \begin{equation}\label{eq:quant}
    \|\bra{\cdot}^{-\mu}*|\phi|(|x|\theta)\|
       _{L^{\q}_{\theta}(\mathbb{S}^{n-1})}\lesssim
    \int_{0}^{\infty}
       J_{\mu\R}(|x|,\rho)^{\frac{1}{\R}}
    \|\phi(\rho \theta)\|
       _{L^{\p}_{\theta}(\mathbb{S}^{n-1})}\rho^{n-1} d\rho.
  \end{equation}
  Notice that when $\R=\infty$ we have
  \begin{equation*}
    \|\bra{|x|e-\rho \theta}^{-\mu}\|_{L^{\infty}_{\theta}}\lesssim
    \bra{|x|-\rho}^{-\mu}.
  \end{equation*}
  We write for brevity
  \begin{equation*}
    Q(|x|)\equiv
    |x|^{-\beta+\frac{n-1}{q}}
    \|\bra{\cdot}^{-\mu}*|\phi|(|x|\theta)\|
       _{L^{\q}_{\theta}},\qquad
    P(\rho)=\rho^{\alpha+\frac{n-1}{p}}
      \|\phi(\rho \theta)\|
         _{L^{\p}_{\theta}}
  \end{equation*}
  \begin{equation*}
    J(|x|,\rho)=J_{\mu\R}^{\frac{1}{\R}}(x,\rho)
    \quad\text{(resp. $\bra{|x|-\rho}^{-\mu}$ if $\R=\infty$).}
  \end{equation*}
  Thus \eqref{eq:quant} becomes
  \begin{equation}\label{eq:quantt}
    Q(\sigma)\lesssim
    \sigma^{-\beta+\frac{n-1}{q}}
    \int_{0}^{\infty}J(\sigma,\rho)
    \rho^{\frac{n-1}{p'}-\alpha}
    P(\rho)
    d\rho
  \end{equation}
  and the estimate to be proved \eqref{eq:estmu}
  can be written as
  \begin{equation}\label{eq:thesis}
    \|Q\|_{L^{q}(0,+\infty)}\lesssim
    \|P\|_{L^{p}(0,+\infty)}
  \end{equation}
  Recall that the integrals of the form $J(\sigma,\rho)$ have
  been estimated in Lemma \ref{lem:singint2}.
  
  We split $Q$ into the sum of several terms corresponding
  to different regions of $\rho,\sigma$.
  In the region $\sigma\le1$ we have 
  $J(\sigma,\rho)\lesssim \bra{\rho}^{-\mu}$
  so that
  \begin{equation}\label{eq:quant2}
    Q_{1}(\sigma)\lesssim
    \sigma^{-\beta+\frac{n-1}{q}}
    \int_{0}^{\infty}\bra{\rho}^{-\mu}\rho^{\frac{n-1}{p'}-\alpha}
    P(\rho)
    d \rho
  \end{equation}
  Thus we see that in this region \eqref{eq:thesis} follows simply from
  H\"older's inequality and the fact that $\alpha<n/p'$
  and $\beta<n/q$. Similarly, it is easy to handle the part of the
  integral with $\rho\le1$ since we have then
  $J(\sigma,\rho)\lesssim \bra{\sigma}^{-\mu}$. Thus in the
  following we can restrict to $\sigma \gtrsim 1,\rho \gtrsim 1$.

  When $1 \lesssim \sigma \le \rho/2$ we have again 
  $J(\sigma,\rho)\lesssim \bra{\rho}^{-\mu}$ 
  and \eqref{eq:quantt} becomes
  \begin{equation}\label{eq:quant3}
    Q_{2}(\sigma)\lesssim
    \sigma^{-\beta+\frac{n-1}{q}}
    \int_{\sigma}^{\infty}\bra{\rho}^{-\mu}\rho^{\frac{n-1}{p'}-\alpha}
    P(\rho)
    d \rho
  \end{equation}
  If we assume
  \begin{equation}\label{eq:conda}
    \mu>\frac{n}{p'}-\alpha
  \end{equation}
  we can apply H\"older's inequality and we get
  \begin{equation*}
    Q_{3}(\sigma)\lesssim
    \sigma^{-\beta+\frac{n-1}{q}}
    \sigma^{\frac{n}{p'}-\mu-\alpha}\|P\|_{L^{p}}.
  \end{equation*}
  Now the right hand side is in $L^{q}(\sigma\ge1)$ provided
  \begin{equation}\label{eq:condb}
    \mu>\frac{n}{p'}-\alpha+\frac nq-\beta \equiv
    -\alpha-\beta+n\left(1+\frac1q-\frac1p\right)
  \end{equation}
  and we see that \eqref{eq:condb} implies \eqref{eq:conda}
  since $\beta<n/q$ by assumption.

  When $1 \lesssim \rho \le \sigma/2$ we have 
  $J(\sigma,\rho)\lesssim \bra{\sigma}^{-\mu}$ 
  and \eqref{eq:quantt}
  becomes
  \begin{equation}\label{eq:quant3}
    Q_{4}(\sigma)\lesssim
    \sigma^{-\beta+\frac{n-1}{q}}\sigma^{-\mu}
    \int_{0}^{\sigma}\rho^{\frac{n-1}{p'}-\alpha}
    P(\rho)
    d \rho
  \end{equation}
  and by H\"older's inequality we have as before
  \begin{equation*}
    \lesssim
    \sigma^{-\beta+\frac{n-1}{q}}
    \sigma^{\frac{n}{p'}-\mu-\alpha}
    \|P\|_{L^{p}}
  \end{equation*}
  so that \eqref{eq:condb} is again sufficient to obtain 
  \eqref{eq:thesis}.

  Finally, let $\sigma \gtrsim 1$, $\rho \gtrsim 1$ and
  $2^{-1}\sigma\le \rho\le 2 \sigma$.
  In this region we must treat differently the values of
  $\mu\R$ larger or smaller than $n-1$, and the case
  $\R=\infty$ is considered at the end.
  Assume that $n-1<\mu\R\le n$; then
  $J(\sigma,\rho)\lesssim \bra{\rho}^{1-n}
  \bra{\sigma-\rho}^{\frac{n-1}{\R}-\mu}$,
  and using the relations
  \begin{equation*}
    \sigma \simeq \rho,\qquad
    \frac1\R=1+\frac1\p-\frac1\q
  \end{equation*}
  we see that \eqref{eq:quantt} reduces to
  \begin{equation}\label{eq:quant4}
    Q_{5}(\sigma)\lesssim
    \sigma^{-\alpha-\beta+(n-1)
      \left(
        \frac1q-\frac1p+\frac1\p-\frac1\q
      \right)}
    \int_{\sigma/2}^{2\sigma}
    \bra{\sigma-\rho}^{\frac{n-1}{\R}-\mu}
    P(\rho)
    d \rho.
  \end{equation}
  The last integral is (bounded by)
  a convolution of $P(\rho)$ with the function
  $\bra{\rho}^{\frac{n-1}{\R}-\mu}$.
  In order to estimate the $L^{q}(\sigma\ge1)$ norm of $Q_{5}$,
  we use first H\"older's then Young's inequality:
  \begin{equation*}
    \|Q_{5}\|_{L^{q}}\lesssim
    \|\bra{\sigma}^{-\epsilon}\|_{L^{q_{0}}}
    \|\bra{\rho}^{\frac{n-1}{\R}-\mu}\|_{L^{q_{1}}}
    \|P\|_{L^{p}}
  \end{equation*}
  where
  \begin{equation*}
    \epsilon=-\alpha-\beta+(n-1)
      \left(
        \frac1q-\frac1p+\frac1\p-\frac1\q
      \right),\qquad
    \frac1q=\frac{1}{q_{0}}+\frac{1}{q_{1}}+\frac1p-1.
  \end{equation*}
  By assumption we have $\epsilon\ge0$.
  When $\epsilon>0$,
  in order for the norms to be finite we need
  \begin{equation*}
    \epsilon q_{0}> 1,\qquad
    \frac{n-1}{\R}-\mu<-\frac{1}{q_{1}}
  \end{equation*}
  which can be rewritten
  \begin{equation*}
    (n-1)\left(
      1+\frac1\q-\frac1\p
    \right)-\mu +1+\frac1q-\frac1p
    < \frac1{q_{0}}
    <\epsilon
  \end{equation*}
  and we see that we can find a suitable $q_{0}$ provided
  the first side is strictly smaller than the last side; this
  condition is precisely equivalent to \eqref{eq:condb} again
  (recall also that $n-1<\mu\le n$).
  The argument works also in the case $\epsilon=0$ by choosing
  $q_{0}=\infty$.

  If on the other hand $0<\mu<n-1$, we have
  $J_{\mu\R}^{\frac1\R}\lesssim \bra{\rho}^{-\mu}$ also
  in this region, so that
  \begin{equation*}
    Q_{5}(\sigma)\lesssim
    \sigma^{-\beta+\frac{n-1}{q}}
    \sigma^{\frac{n-1}{p'}-\alpha-\mu}
    \int_{\sigma/2}^{\sigma}
    P(\rho)
    d \rho
  \end{equation*}
  by $\sigma \simeq \rho$. H\"older's inequality gives
  \begin{equation*}
    Q_{5}(\sigma)\lesssim
    \sigma^{-\beta+\frac{n-1}{q}}
    \sigma^{\frac{n-1}{p'}-\alpha-\mu}
    \sigma^{\frac1{p'}}\|P\|_{L^{p}}
  \end{equation*}
  which leads to exactly the same computations as above and in the
  end to \eqref{eq:condb}. The case $\mu=n-1$ introduces a
  logarithmic term which does not change the integrability properties
  used here.  
  
  It remains the last region when $\R=\infty$ so that
  $J(\sigma,\rho)=\bra{\sigma-\rho}^{-\mu}$
  and $1/\p-1/\q=1$. Then
  \begin{equation*}
    Q_{5}(\sigma)\lesssim 
    \sigma^{-\beta+\frac{n-1}{q}}
    \int_{\sigma/2}^{2 \sigma}\bra{\sigma-\rho}^{-\mu}
    \rho^{\frac{n-1}{p'}-\alpha}
    P(\rho)d \rho
  \end{equation*}
  which is identical with \eqref{eq:quant4} with $\R=\infty$,
  thus the same computations apply and the proof is concluded.
\end{proof}

\section{Proof of Theorem \ref{the:Our2Thm}}\label{sec:proof_of2} 

The Caffarelli-Kohn-Niremberg inequality is a simple corollary
of Theorem \eqref{the:Our1Thm}. We begin by taking
$0<a\le1$, and indices $r,\R,s,\s,q,\q\in[1,+\infty]$ such that
\begin{equation}\label{eq:rsq}
  \frac1r=\frac as+\frac{1-a}q,\qquad
  \frac1\R=\frac a\s+\frac{1-a}\q.
\end{equation}
Then by two applications of H\"older's inequality we obtain
the interpolation inequality
\begin{equation}\label{eq:holder}
\begin{split}
  \||x|^{-\gamma}u\|_{L^{r}L^{\R}}
  =&
  \|(|x|^{-\delta}u)^{a}(|x|^{-\beta}u)^{1-a}\|_{L^{r}L^{\R}}
    \\
  \le&
  \|(|x|^{-\delta}u)^{a}\|_{L^{s/a}L^{\s/a}}
  \|(|x|^{-\beta}u)^{1-a}\|_{L^{q/(1-a)}L^{\q/(1-a)}}
    \\
  =&
  \||x|^{-\delta}u\|^{a}_{L^{s}L^{\s}}
  \||x|^{-\beta}u\|^{1-a}_{L^{q}L^{\q}}
\end{split}
\end{equation}
provided the exponents $\gamma,\delta,\beta$ are related by
\begin{equation}\label{eq:expon}
  \gamma=a \delta+(1-a)\beta.
\end{equation}
Now the main step of the proof. By Theorem \eqref{the:Our1Thm}
we know that
\begin{equation*}
  \||x|^{-\delta}T_{\lambda}u\|_{L^{s}L^{\s}}\lesssim
  \||x|^{-\alpha}u\|_{L^{p}L^{\p}}
\end{equation*}
under suitable conditions on the indices. 
Now using the well known estimate
\begin{equation}\label{eq:fractest}
  |u(x)|\le C_{\lambda,n}
    T_{\lambda}\left(
      \left|
        |D|^{n-\lambda}u
      \right|
    \right)
\end{equation}
the previous inequality can be equivalently written
\begin{equation*}
  \||x|^{-\delta}u\|_{L^{s}L^{\s}}\lesssim
  \||x|^{-\alpha}|D|^{\sigma}u\|_{L^{p}L^{\p}},\qquad
  \sigma=n-\lambda
\end{equation*}
which together with \eqref{eq:holder} gives
\begin{equation}\label{eq:final}
  \||x|^{-\gamma}u\|_{L^{r}L^{\R}}\lesssim
  \||x|^{-\alpha}|D|^{\sigma}u\|_{L^{p}L^{\p}}^{a}
  \||x|^{-\beta}u\|^{1-a}_{L^{q}L^{\q}}.
\end{equation}
The conditions on the indices are those given by \eqref{eq:rsq},
\eqref{eq:expon}, plus those listed in the statement of
Theorem \eqref{the:Our1Thm} (notice that we are using $-\alpha$
instead of $\alpha$). The complete list is the following:
\begin{equation}\label{eq:tot1}
  r,s,q,\R,\s,\q\in[1,+\infty],\qquad 
  a<0\le1,\qquad
  0<\sigma<n,
\end{equation}
\begin{equation}\label{eq:tot2}
  \frac1r=\frac as+\frac{1-a}q,\qquad
  \frac1\R=\frac a\s+\frac{1-a}\q.
\end{equation}
\begin{equation}\label{eq:tot3}
  1<s\le p<\infty,\qquad
  1\le\s\le\p\le \infty,
\end{equation}
\begin{equation}\label{eq:tot4}
  \gamma<\frac nr,\qquad
  \beta<\frac nq,\qquad
  -\alpha<\frac n{p'},\qquad
  \delta<\frac ns,
\end{equation}
\begin{equation}\label{eq:tot5}
  \gamma=a \delta+(1-a)\beta,
\end{equation}
\begin{equation}\label{eq:tot6}
  -\alpha+\delta+n-\sigma=n+\frac ns-\frac np,
\end{equation}
\begin{equation}\label{eq:tot7}
  -\alpha+\delta\ge
  (n-1)\left(\frac1s-\frac1p+\frac1\p-\frac1\s\right).
\end{equation}
Recall also that, when the last inequality \eqref{eq:tot7}
is strict, we can allow the full range
\begin{equation*}
  1\le s\le p\le \infty.
\end{equation*}
Our final task is to rewrite this set of conditions in a
compact form, eliminating the redundant parameters
$\delta,s,\s$. Define the two quantities
\begin{equation*}
  \Delta=a \sigma+n
    \left(
      \frac1r-\frac{1-a}{q}-\frac ap
    \right),\qquad
  \widetilde{\Delta}=a \sigma+n
    \left(
      \frac1\R-\frac{1-a}{\q}-\frac a\p
    \right).
\end{equation*}
Then \eqref{eq:tot2} are equivalent to
\begin{equation}\label{eq:tot2b}
  \Delta=a \left(\sigma+\frac ns-\frac np\right),\qquad
  \widetilde{\Delta}=a \left(\sigma+\frac n\s-\frac n\p\right)
\end{equation}
while \eqref{eq:tot6} is equivalent to
\begin{equation}\label{eq:tot6b}
  \delta=\alpha+\frac{\Delta}{a}
\end{equation}
and we can use \eqref{eq:tot2b}, \eqref{eq:tot6b} to replace
$\delta,s,\s$ in the remaining relations. Condition \eqref{eq:tot5}
becomes
\begin{equation}\label{eq:tot5b}
  \Delta=\gamma-a \alpha-(1-a)\beta,
\end{equation}
which is precisely the scaling condition,
while \eqref{eq:tot7} becomes
\begin{equation}\label{eq:tot7b}
  \Delta+(n-1)\widetilde{\Delta}\ge0.
\end{equation}
The last inequality in \eqref{eq:tot4}, $\delta<n/s$, can be written
\begin{equation*}%
  \alpha<\frac np-\sigma
\end{equation*}
so that \eqref{eq:tot4} is replaced by
\begin{equation}\label{eq:tot4b}
  \gamma<\frac nr,\qquad
  \beta<\frac nq,\qquad
  \frac np-n<\alpha<\frac np-\sigma.
\end{equation}
Finally, conditions \eqref{eq:tot3} translate to
\begin{equation}\label{eq:tot3b}
  1<p,\qquad
  a\left(\sigma-\frac np\right)<\Delta\le a \sigma, \qquad 
  a\left(\sigma-\frac n\p\right)\le \widetilde{\Delta}\le a \sigma.
\end{equation}
When the inequality in \eqref{eq:tot7b} is strict, the last
condition can be relaxed to
\begin{equation}\label{eq:tot3b2}
  1\le p,\qquad
  a\left(\sigma-\frac np\right)\le\Delta\le a \sigma, \qquad 
  a\left(\sigma-\frac n\p\right)\le \widetilde{\Delta}\le a \sigma.
\end{equation}

We pass now to the proof of Corollary \ref{cor:integers}.
Assume now $\sigma$ is integer, and the inequality
\begin{equation*}
  \||x|^{-\gamma}u\|_{L^{r}L^{\R}}\le C
  \||x|^{-\alpha}|D|^{\sigma} u\|^{a}_{L^{p}L^{\p}}
  \||x|^{-\beta}u\|^{1-a}_{L^{q}L^{\q}}
\end{equation*}
is true for a certain choice of the parameters as in the theorem,
so that in particular
\begin{equation*}
  \alpha< \frac np-\sigma<\frac np.
\end{equation*}
Then we shall prove that also the following inequalities are true
\begin{equation}\label{eq:goalk}
  \||x|^{k-\gamma}u\|_{L^{r}L^{\R}}\le C
  \||x|^{k-\alpha}D^{\sigma} u\|^{a}_{L^{p}L^{\p}}
  \||x|^{k-\beta}u\|^{1-a}_{L^{q}L^{\q}}
\end{equation}
for all integers $k\ge0$, where we are using the
shorthand notation
\begin{equation*}
  \||x|^{k-\alpha}D^{\sigma} u\|_{L^{p}L^{\p}}=
  \sum_{|\nu|=\sigma}
  \||x|^{k-\alpha}D^{\nu} u\|_{L^{p}L^{\p}},\qquad
  (\nu=(\nu_{1},\dots,\nu_{n})\in \mathbb{N}^{n}).
\end{equation*}
This in particular implies that the condition on $\alpha$ from
below can be dropped when $\sigma$ is integer.

When $k=0$, \eqref{eq:goalk} is obtained just by
replacing $|D|^{\sigma}$ with $D^{\sigma}$ in the 
original inequality. 
The proof of this estimate is identical to the previous one;
the only modification is to use, instead of \eqref{eq:fractest},
the stronger pointwise bound
\begin{equation}\label{eq:fractest2}
  |u(x)|\le C_{\lambda,n}
    T_{\lambda}\left(
      \left|
        D^{n-\lambda}u
      \right|
    \right)
\end{equation}
which is valid for all $\lambda=1,\dots,n-1$.

Now if we apply \eqref{eq:goalk} (with $k=0$) to a function of the
form $|x|^{k}u$ for some $k\ge1$, we obtain
\begin{equation*}
  \||x|^{k-\gamma}u\|_{L^{r}L^{\R}}\le C
  \||x|^{-\alpha}D^{\sigma}(|x|^{k}u)\|^{a}_{L^{p}L^{\p}}
  \||x|^{k-\beta}u\|^{1-a}_{L^{q}L^{\q}}
\end{equation*}
and to conclude the proof we see that it is sufficient to prove the
inequality
\begin{equation}\label{eq:interm}
  \||x|^{-\alpha}D^{\sigma}(|x|^{k}u)\|_{L^{p}L^{\p}}\lesssim
  \||x|^{k-\alpha}D^{\sigma} u\|_{L^{p}L^{\p}}
\end{equation}
for all $\alpha<n/p$, $1\le p,\p<\infty$,
and integers $\sigma=1,\dots,n-1$, $k\ge1$. Notice indeed that all
the conditions on the parameters 
(apart from $\alpha>-n+n/p$)
are unchanged if we decrease
$\gamma,\alpha,\beta$ by the same quantity.

By induction on $k$ (and writing $\delta=-\alpha$),
we are reduced to prove that
for all $p,\p\in[1,\infty)$ and $1\le\sigma\le n-1$
\begin{equation}\label{eq:interm2}
  \||x|^{\delta}D^{\sigma}(|x|u)\|_{L^{p}L^{\p}}\lesssim
  \||x|^{1+\delta}D^{\sigma} u\|_{L^{p}L^{\p}},\qquad
  \delta>\sigma-\frac np.
\end{equation}
Using Leibnitz' rule we reduce further to
\begin{equation}\label{eq:interm3}
  \||x|^{1+\delta-\ell}u\|_{L^{p}L^{\p}}\lesssim
  \||x|^{1+\delta}D^{\ell} u\|_{L^{p}L^{\p}},\qquad
  \delta>\ell-\frac np
\end{equation}
for $\ell=1,\dots,n-1$, and by induction on $\ell$
this is implied by
\begin{equation}\label{eq:interm4}
  \||x|^{\delta}u\|_{L^{p}L^{\p}}\lesssim
  \||x|^{1+\delta}\nabla u\|_{L^{p}L^{\p}},\qquad
  \delta>1-\frac np.
\end{equation}
In order to prove \eqref{eq:interm4}, consider first
the radial case.
When $u=\phi(|x|)$ is a radial (smooth compactly supported)
function, we have
\begin{equation*}
  \||x|^{\delta}u\|_{L^{p}L^{\p}}^{p}\simeq
  \int_{0}^{\infty}\rho^{\delta p+n-1}|\phi(\rho)|^{p}d\rho.
\end{equation*}
Integrating by parts we get
\begin{equation*}
\begin{split}
  =&-\frac{p}{\delta p+n}\int_{0}^{\infty}
    \rho^{\delta p+n}|\phi|^{p-1}|\phi(\rho)|'d\rho
    \\
  \lesssim &
  \int_{0}^{\infty}(\rho^{\delta p+n-1}|\phi|^{p})^{\frac{p-1}{p}}
  (\rho^{\delta p+p+n-1}|\phi'|^{p})^{\frac1p}d\rho
    \\
  \simeq &
  \||x|^{\delta}u\|^{\frac{p-1}{p}}_{L^{p}L^{\p}}
  \||x|^{1+\delta}\nabla u\|_{L^{p}L^{\p}}
\end{split}
\end{equation*}
which implies \eqref{eq:interm4} in the radial case. If $u$ is not
radial, define
\begin{equation*}
  \phi(\rho)=\|u(\rho \theta)\|_{L^{\p}_{\theta}(\mathbb{S}^{n-1})}
  =
  \left(
    \int_{\mathbb{S}^{n-1}}
    |u(\rho \theta)|^{\p}dS_{\theta}
  \right)^{\frac1\p}
\end{equation*}
so that
\begin{equation*}
  \||x|^{\delta}u\|_{L^{p}L^{\p}}\simeq
  \left(\int_{0}^{\infty}\rho^{\delta p+n-1}|\phi(\rho)|^{p}d\rho
  \right)
  ^{\frac1p}.
\end{equation*}
The proof in the radial case implies
\begin{equation*}
  \||x|^{\delta}u\|_{L^{p}L^{\p}}\le
  \||x|^{\delta+1}\phi'(|x|)\|_{L^{p}};
\end{equation*}
moreover we have
\begin{equation*}
\begin{split}
  |\phi'(\rho)|\lesssim &\ 
  \phi^{1-\p}
  \int_{\mathbb{S}^{n-1}}
  |u(\rho \theta)|^{\p-1}|\theta \cdot \nabla u|\ 
  dS_{\theta}
    \\
  \le &\ 
  \phi^{1-\p}
  \left(\int_{\mathbb{S}}|u|^{\p}\right)^{\frac{\p-1}\p}
  \left(\int_{\mathbb{S}}|\nabla u|^{\p}\right)^{\frac1\p}
  =\|\nabla u(\rho \theta)\|_{L^{\p}_{\theta}(\mathbb{S}^{n-1})}
\end{split}
\end{equation*}
and in conclusion we obtain
\begin{equation*}
  \||x|^{\delta}u\|_{L^{p}L^{\p}}\le
  \||x|^{\delta+1}\nabla u\|_{L^{p}L^{\p}}
\end{equation*}
as claimed.


\begin{thebibliography}{10}

\bibitem{CaffarelliKohnNirenberg84-a}
Luis~A. Caffarelli, Robert Kohn, and Louis Nirenberg.
\newblock First order interpolation inequalities with weights.
\newblock {\em Compositio Math.}, 53(3):259--275, 1984.

\bibitem{ChoOzawa09-a}
Yonggeun Cho and Tohru Ozawa.
\newblock Sobolev inequalities with symmetry.
\newblock {\em Commun. Contemp. Math.}, 11(3):355--365, 2009.

\bibitem{DanconaCacciafesta11-a}
Piero D'Ancona and Federico Cacciafesta.
\newblock Endpoint estimates and global existence for the nonlinear dirac
  equation with potential.
\newblock (preprint 2011).

\bibitem{DenapoliDrelichmanDuran09-a}
Pablo~L. De~N{\'a}poli, Irene Drelichman, and Ricardo~G. Dur{\'a}n.
\newblock Radial solutions for {H}amiltonian elliptic systems with weights.
\newblock {\em Adv. Nonlinear Stud.}, 9(3):579--593, 2009.

\bibitem{DenapoliDrelichmanDuran11-a}
Pablo~L. De~N{\'a}poli, Irene Drelichman, and Ricardo~G. Dur{\'a}n.
\newblock On weighted inequalities for fractional integrals of radial
  functions.
\newblock {\em Illinois J. Math.}, 2011.

\bibitem{FangWang08-a}
Daoyuan Fang and Chengbo Wang.
\newblock Weighted {S}trichartz estimates with angular regularity and their
  applications.
\newblock 2008.

\bibitem{GinibreVelo95-b}
Jean Ginibre and Giorgio Velo.
\newblock Generalized {S}trichartz inequalities for the wave equation.
\newblock {\em J. Funct. Anal.}, 133(1):50--68, 1995.

\bibitem{Grafakos08-a}
Loukas Grafakos.
\newblock {\em Classical {F}ourier analysis}, volume 249 of {\em Graduate Texts
  in Mathematics}.
\newblock Springer, New York, second edition, 2008.

\bibitem{HidanoKurokawa08-a}
Kunio Hidano and Yuki Kurokawa.
\newblock Weighted {HLS} inequalities for radial functions and {S}trichartz
  estimates for wave and {S}chr{\"o}dinger equations.
\newblock {\em Illinois J. Math.}, 52(2):365--388, 2008.

\bibitem{JiangWangYu10-a}
Chengbo~Wang Jin-Cheng~Jiang and Xin Yu.
\newblock Generalized and weighted strichartz estimates.
\newblock 2010.

\bibitem{KeelTao98-a}
Markus Keel and Terence Tao.
\newblock Endpoint {S}trichartz estimates.
\newblock {\em Amer. J. Math.}, 120(5):955--980, 1998.

\bibitem{MachiharaNakamuraNakanishi05-a}
Shuji Machihara, Makoto Nakamura, Kenji Nakanishi, and Tohru Ozawa.
\newblock Endpoint {S}trichartz estimates and global solutions for the
  nonlinear {D}irac equation.
\newblock {\em J. Funct. Anal.}, 219(1):1--20, 2005.

\bibitem{Ni82-a}
Wei~Ming Ni.
\newblock A nonlinear {D}irichlet problem on the unit ball and its
  applications.
\newblock {\em Indiana Univ. Math. J.}, 31(6):801--807, 1982.

\bibitem{SawyerWheeden92-a}
Eric~T. Sawyer and Richard~L. Wheeden.
\newblock Weighted inequalities for fractional integrals on {E}uclidean and
  homogeneous spaces.
\newblock {\em Amer. J. Math.}, 114(4):813--874, 1992.

\bibitem{SickelSkrzypczak00-a}
Winfried Sickel and Leszek Skrzypczak.
\newblock Radial subspaces of {B}esov and {L}izorkin-{T}riebel classes:
  extended {S}trauss lemma and compactness of embeddings.
\newblock {\em J. Fourier Anal. Appl.}, 6(6):639--662, 2000.

\bibitem{Stein93-a}
Elias~M. Stein.
\newblock {\em Harmonic analysis: real-variable methods, orthogonality, and
  oscillatory integrals}, volume~43 of {\em Princeton Mathematical Series}.
\newblock Princeton University Press, Princeton, NJ, 1993.
\newblock With the assistance of Timothy S. Murphy, Monographs in Harmonic
  Analysis, III.

\bibitem{SteinWeiss58-b}
Elias~M. Stein and Guido Weiss.
\newblock Fractional integrals on {$n$}-dimensional {E}uclidean space.
\newblock {\em J. Math. Mech.}, 7:503--514, 1958.

\bibitem{Sterbenz05-a}
Jacob Sterbenz.
\newblock Angular regularity and {S}trichartz estimates for the wave equation.
\newblock {\em Int. Math. Res. Not.}, (4):187--231, 2005.
\newblock With an appendix by Igor Rodnianski.

\bibitem{Strauss77-a}
Walter~A. Strauss.
\newblock Existence of solitary waves in higher dimensions.
\newblock {\em Comm. Math. Phys.}, 55(2):149--162, 1977.

\bibitem{Vilela01-a}
Maria~Cruz Vilela.
\newblock Regularity of solutions to the free {S}chr{\"o}dinger equation with
  radial initial data.
\newblock {\em Illinois J. Math.}, 45(2):361--370, 2001.

\end{thebibliography}
\end{document}